\documentclass[10pt,fleqn]{article}

\input{epsf.tex}
\usepackage{latexsym,amsfonts,amssymb,epsfig,verbatim}
\usepackage{amsmath,amsthm,amssymb,latexsym,textcomp}
\usepackage{eucal}
\usepackage{graphicx}
\usepackage{psfrag, epsf}
\usepackage{url}

\theoremstyle{plain}
\newtheorem{theorem}{Theorem}[section]
\newtheorem{lemma}[theorem]{Lemma}

\newtheorem{definition}[theorem]{Definition}

\newtheorem{proposition}[theorem]{Proposition}

\newtheorem*{claim}{Claim}

\theoremstyle{definition}
\newtheorem*{remark}{Remark}

\newtheorem*{acknowledgements}{Acknowledgements}

\def\1{\mathbf 1}

\def\Mod{\mathrm{Mod}}

\begin{document}

\title{\textbf{The asymptotic behavior of least pseudo-Anosov dilatations}}
\author{Chia-yen Tsai}
\maketitle
{\abstract For a surface $S$ with $n$ marked points and fixed genus $g\geq2$, we prove that the logarithm of the minimal dilatation of a pseudo-Anosov homeomorphism of $S$ is on the order of $\frac{\log n}{n}$.  This is in contrast with the cases of genus zero or one where the order is $\frac{1}{n}$.}

\section{Introduction}\label{introsect}
Let $S=S_{g,n}$ be an orientable surface with genus $g$ and $n$ marked points.  The {\bf mapping class group} of $S$ is defined to be the group of homotopy classes of orientation preserving homeomorphisms of $S$.  We denote it as $\Mod(S)$.  Given $f \in \Mod(S)$ a pseudo-Anosov element, let $\lambda(f)$ denote the {\bf dilatation} of $f$; see section 2.1.  We define
\begin{align*}
{\cal L}(S_{g,n}):=\{ \log \lambda(f) | f\in \Mod(S_{g,n}) \text{ pseudo-Anosov } \}.
\end{align*}
This is precisely the length spectrum of the moduli space ${\cal M}_{g,n}$ of Riemann surfaces of genus $g$ with $n$ marked points with respect to the Teichmuller metric; see \cite{Iv}.  There is a shortest closed geodesic and we denote its length
\begin{align*}
 l_{g,n}=\min\{\log \lambda(f)|f \in \Mod(S_{g,n})\text{ pseudo-Anosov} \}.
\end{align*}
Our main theorem is the following:
\begin{theorem}\label{main}
For any fixed $g\geq 2$, there is a constant $c_{g}\geq1$ depending on $g$ such that
\begin{align*}
\frac{\log n}{c_{g}n} < l_{g,n} < \frac{c_{g}\log n}{n},
\end{align*}
for all $n\geq 3$.
\end{theorem}

To contrast with known results, we recall that in \cite{Pe}, Penner proves for $2g-2+n>0$ and $n\geq0$,
\begin{align*}
l_{g,n} \geq \frac{\log2}{12g-12+4n},
\end{align*}
and for closed surfaces
\begin{align*}
\frac{\log2}{12g-12}\leq l_{g,0} \leq \frac{\log11}{g}.
\end{align*}
The bounds on $l_{g,0}$ have been improved by a number of authors; \cite{Ba}, 
\cite{Mc}, 
\cite{Mk},
\cite{HK}.

In \cite{Pe}, Penner suggests that there may be an ``analogous upper bound for $n\neq0$''.  In \cite{HK}, Hironaka and Kin use a concrete construction to prove that for genus $g=0$
\begin{align*}
l_{0,n}< \frac{\log(2+\sqrt{3})}{ \left\lfloor \frac{n-2}{2} \right\rfloor }\leq \frac{2\log(2+\sqrt{3})}{n-3},
\end{align*}
for all $n\geq 4$.  The inequality is proven for even $n$ in \cite{HK}, but it follows for odd $n$ by letting the fixed point of their example be a marked point.  Combining this with Penner's lower bound, one sees for $n\geq 4$
\begin{align*}
\frac{\log2}{4n-12} \leq l_{0,n} <\frac{2\log(2+\sqrt{3})}{n-3},
\end{align*}
which shows that the upper bound is on the same order as Penner's lower bound for $g=0$.  A similar situation holds for $g=1$; see the Appendix.

Inspired by the construction of Hironaka and Kin, we tried to find examples of pseudo-Anosov $f_{g,n}\in \Mod(S_{g,n})$ with $\log\lambda(f_{g,n}) = O(\frac{1}{g+n})$.  However for any fixed $g\geq2$, all attempts resulted in $f_{g,n}\in \Mod(S_{g,n})$ pseudo-Anosov with $\log\lambda(f_{g,n}) = O_{g}(\frac{\log n}{n})$ and not $O(\frac{1}{c(g,n)})$ for any linear function $c(g,n)$.  This led us to prove Theorem \ref{main}.

\subsection{Outline of the paper}
We will first recall some definitions and properties in section \ref{pre}.  In section \ref{lowerbd} we prove the lower bound of Theorem \ref{main}.  We construct examples in section \ref{upperbd} which give an upper bound for the genus $2$ case, and we extend the example to arbitrary genus $g\geq 2$ to obtain the upper bound of Theorem \ref{main}.  Finally, we construct a pseudo-Anosov element in $\Mod(S_{1,2n})$ and obtain an upper bound on $l_{1,2n}$ in the Appendix.

\begin{acknowledgements}
The author would like to thank Christopher Leininger for key discussions and for revising an earlier draft.  Kasra Rafi and A.J. Hildebrand offered helpful suggestions and insights.  I would also like to thank MSRI for its stimulating, collaborative research environment during its fall 2007 programs.

\end{acknowledgements}
\section{Preliminaries}\label{pre}

\subsection{Homeomorphisms of a surface}

We say that a homeomorphism $f:S \rightarrow S$ is {\bf pseudo-Anosov} if there are transverse singular foliations ${\cal F}^{s}$ and ${\cal F}^{u}$ together with transverse measures $\mu^{s}$ and $\mu^{u}$ such that for some $\lambda>1$
\begin{align*}
f({\cal F}^{s},\mu^{s})&=({\cal F}^{s},\lambda\mu^{s}),\\
f({\cal F}^{u},\mu^{u})&=({\cal F}^{u},\lambda^{-1}\mu^{u}).
\end{align*}
The number $\lambda=\lambda(f)$ is called the {\bf dilatation} of $f$.  We call $f$ {\bf reducible} if there is a finite disjoint union $U$ of simple essential closed curves on $S$ such that $f$ leaves $U$ invariant.  If there exists $k>0$ such that $f^{k}$ is the identity, then $f$ is {\bf periodic}.  A mapping class $[f]$ is pseudo-Anosov, reducible or periodic (respectively) if $f$ is homotopic to a pseudo-Anosov, reducible or periodic homeomorphism (respectively).  The following is proved in \cite{FLP}.
\begin{theorem}(Nielsen-Thurston)
A mapping class $[f] \in Mod(S)$ is either periodic, reducible, or pseudo-Anosov.
\end{theorem}
As a slight abuse of notation, we sometimes refer to a mapping class $[f]$ by one of its representatives $f$.
\subsection{Markov partitions}\label{secMarkov}
Suppose $f:S \rightarrow S$ is pseudo-Anosov with stable and unstable measured singular foliations $({\cal F}^{s},\mu^{s})$ and $({\cal F}^{u},\mu^{u})$.  We define a rectangle $R$ to be a map $\rho:I\times I \rightarrow S$ such that $\rho$ is an embedding on the interior, $\rho(point\times I)$ is contained in a leaf of ${\cal F}^{u}$, and $\rho(I\times point)$ is contained in a leaf of ${\cal F}^{s}$.  We denote $\rho(\partial I\times I)$ by $\partial ^{u}R$ and $\rho(I\times \partial I)$ by $\partial ^{s}R$.
\begin{figure}[htbp]
\begin{center}
\psfrag{a}{$\partial^{s}R$}
\psfrag{b}{$\partial^{u}R$}
\psfrag{r}{$R$}
\psfrag{c}{}
\includegraphics[width=0.3\textwidth]{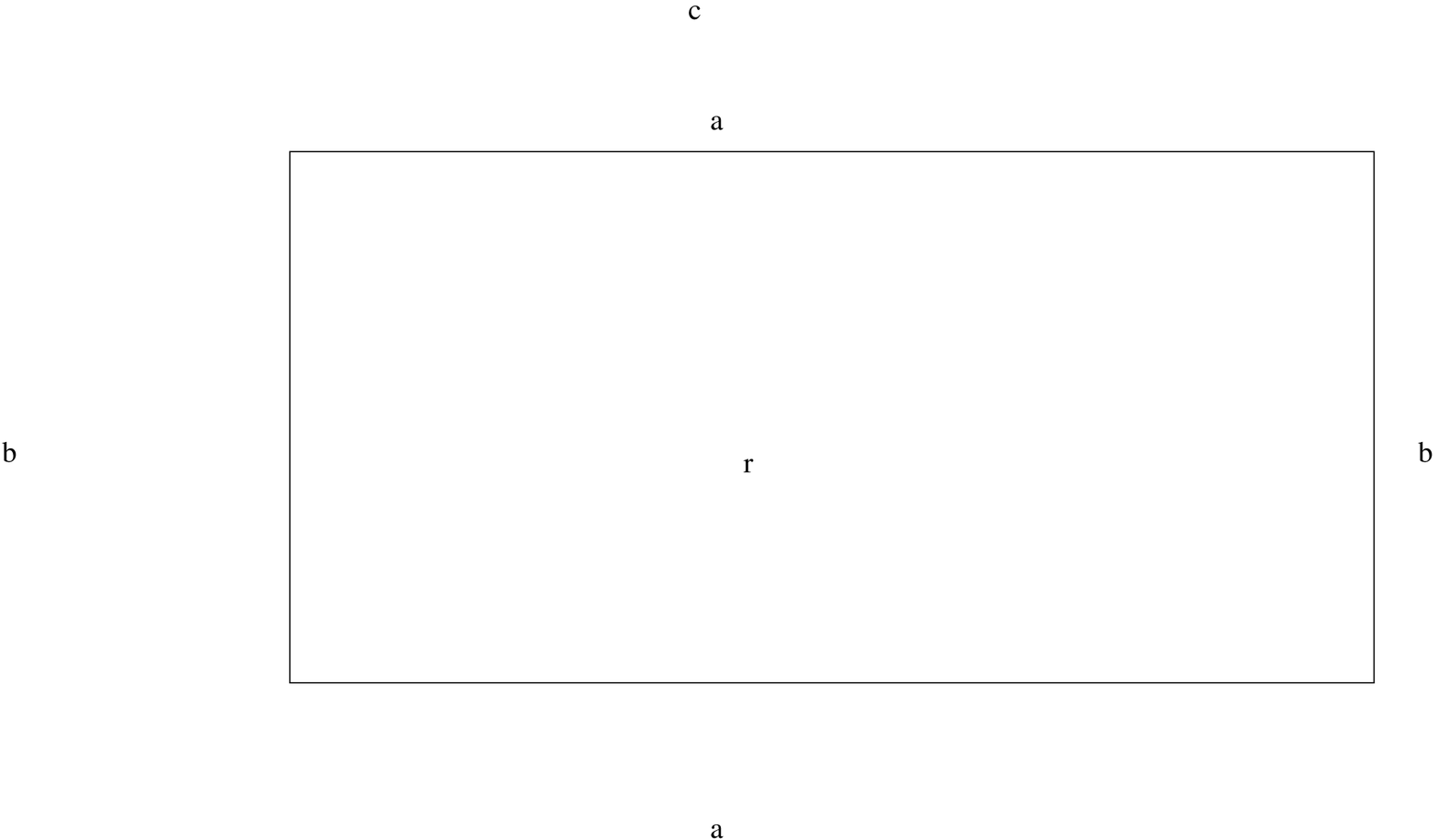}
\end{center}
\end{figure}

As a standard abuse of notation, we will write $R\subset S$ for the image of a rectangle $\rho:I\times I \rightarrow S$.
\begin{definition}
A Markov partition for $f:S \rightarrow S$ is a decomposition of $S$ into a finite union of rectangles $\{R_{i}\}_{i=1}^{k}$, such that:
\begin{enumerate}
\item $Int(R_{i})\cap Int(R_{j})$ is empty,  when $ i\neq j$
\item $f(\bigcup_{j=1}^{k}\partial ^{u}R_{j})\subset \bigcup_{j=1}^{k}\partial ^{u}R_{j}$
\item $f^{-1}(\bigcup_{i=1}^{k}\partial ^{s}R_{i})\subset \bigcup_{i=1}^{k}\partial ^{s}R_{i}$
\end{enumerate}
\end{definition}

Given a pseudo-Anosov homeomorphism $f:S\rightarrow S$, a Markov partition is constructed in \cite{BH} from a train track map for $f$.  The advantage of this construction over \cite{FLP}, for example, is that the number of rectangles is at most the number of branches of the train track.  So since the latter number is bounded above by $-9 \chi (S)-3n$ (see \cite{PeHa}), one has the following.
\begin{theorem}(Bestvina-Handel)\label{markov}
For any pseudo-Anosov homeomorphism $f:S\rightarrow S$, there exists a Markov partition of $f$ with $k$ rectangles, where $k \leq -9 \chi (S)-3n $.
\end{theorem}

We say that a matrix is {\it positive (respectively nonnegative)} if all the entries are positive (respectively nonnegative).

We can define a {\it transition matrix} $M$ associated to the Markov partition with rectangles $\{R_{i}\} _{i=1}^{k}$.  The entry $m_{i,j}$ of $M$ is the number of times that $f(R_{j})$ wraps over $R_{i}$, so $M$ is a nonnegative integral $k \times k$ matrix.  In Bestvina and Handel's construction, $M$ is the same as the transition matrix of the train track map and they show it is an integral Perron-Frobenius matrix (i.e. it is irreducible with nonnegative integer entries); see \cite{Gm}.  Furthermore, the Perron-Frobenius eigenvalue $\mu(M)=\lambda(f)$ is the dilatation of $f$.  The width (respectively height) of $R_{i}$ is the $i$th entry of the corresponding Perron-Frobenius eigenvector of $M$(respectively $M^{T}$), where the eigenvectors are both positive by the irreducibility of $M$. 

The following proposition will be used in proving the lower bound.
\begin{proposition}\label{lbd}
Let $M$ be a $k\times k$ integral Perron-Frobenius matrix.  If there is a nonzero entry on the diagonal of $M$, then $M^{2k}$ is a positive matrix and its Perron-Frobenius eigenvalue $\mu(M^{2k})$ is at least $k$.
\end{proposition}

\begin{proof}
We construct a directed graph $\Gamma$ from $M$ with $k$ vertices $\{i\}_{i=1}^{k}$ such that the number of the directed edge from $i$ to $j$ in $\Gamma$ equals $m_{i,j}$.  We observe that for any $r>0$ the $(i,j)$th entry $m^{(r)}_{i,j}$ of $M^{r}$ is the number of directed edge paths from $i$ to $j$ of length $r$ in $\Gamma$.

Since $M$ is a Perron-Frobenius matrix, we know that $\Gamma$ is path-connected by directed paths.  Suppose $M$ has a nonzero entry at the $(l,l)$th entry, then we will see at least one corresponding loop edge at the vertex $l$.  For any $i$ and $j$ in $\Gamma$, path-connectivity ensures us that there are directed edge paths of length $\leq k$ from $i$ to $l$ and from $l$ to $j$.  This tells us that there is a directed edge path $P$ of length $\leq 2k$ from $i$ to $j$ passing through $l$.  Since we can wrap around the loop edge adjacent to $l$ to increase the length of $P$, there is always a directed edge path of length $2k$ from $i$ to $j$.  In the other words, $m^{(2k)}_{i,j}$ is at least $1$ for all $i$ and $j$, so $M^{2k}$ is a positive matrix.

Let $v$ be the corresponding Perron-Frobenius eigenvector, so that we have $M^{2k}v=\mu(M^{2k})v$.  This implies that if $v=[ v_{1}\cdots v_{k}]^{T}$, for all $i$
\begin{align*}
\sum_{j=1}^{k}m^{(2k)}_{i,j}v_{j}&=\mu(M^{2k})v_{i},
\end{align*}
equivalently,
\begin{align*}
\mu(M^{2k})=\sum_{j=1}^{k}m^{(2k)}_{i,j}\frac{v_{j}}{v_{i}}.
\end{align*}
Choosing $i$ such that $v_{i} \leq v_{j}$ for all $j$, we obtain
\begin{align*}
\mu(M^{2k}) &\geq \sum_{j=1}^{k}m^{(2k)}_{i,j} \geq \sum_{j=1}^{k}1 =k.
\end{align*}
\end{proof}

The following proposition will be used in proving the upper bound.
\begin{proposition}\label{dgraph}
Let $\Gamma$ be the induced directed graph of a integral Perron-Frobenius matrix $M$ with Perron-Frobenius eigenvalue $\mu(M)=\mu$.  The total number of paths of length $d$ from vertex $i$ in $\Gamma$ is denoted by $P_{\Gamma}(i,d)$.  We have
\begin{center}
 $\sqrt[d]{P_{\Gamma}(i,d)}\longrightarrow \mu(M)$ as $d\rightarrow \infty$, $\forall i$.
\end{center}
\end{proposition}
\begin{proof}
Let $M$ be an integral $k\times k$ Perron-Frobenius matrix with Perron-Frobenius eigenvalue $\mu$ and Perron-Frobenius eigenvector $v$.  As above
\begin{align*}
\sum_{j=1}^{k}m^{(d)}_{i,j}v_{j}=\mu(M^{d})v_{i}=\mu^{d}v_{i}.
\end{align*}
Let $v_{max}=\max_{i}\{v_{i} \}$ and $v_{min}=\min_{i}\{v_{i} \}$.  Recall that $v_{i}>0$ for all $i$ by the irreducibility of $M$.  For all $i$ we have
\begin{align*}
\frac{v_{min}\left(\sum_{j} m_{i,j}^{(d)}\right)}{\mu^{d}} &\leq \frac{\sum_{j}m_{i,j}^{(d)}v_{j}}{\mu^{d}}\leq \frac{v_{max}\left(\sum_{j} m_{i,j}^{(d)}\right)}{\mu^{d}}\\
\Rightarrow \frac{v_{i}}{v_{max}}&\leq \frac{\sum_{j} m_{i,j}^{(d)}}{\mu^{d}} \leq\frac{v_{i}}{v_{min}}
\end{align*}
We are done, since $\sum_{j} m_{i,j}^{(d)}=P_{\Gamma}(i,d)$ and for all $i$
\begin{align*}
\sqrt[d]{\frac{v_{i}}{v_{max}}}\rightarrow 1 \text{ and } \sqrt[d]{\frac{v_{i}}{v_{min}}}\rightarrow 1
\text{, as }d \text{ tends to }\infty.
\end{align*}
\end{proof}

\subsection{Lefschetz numbers}
We will review some definitions and properties of Lefschetz numbers.  A more complete discussion can be found in \cite{GP} and \cite{Bott-Tu}.

Let $X$ be a compact oriented manifold, and $f:X \rightarrow X$ be a map.  Define
\begin{align*}
graph(f)=\{(x,f(x))|x\in X\} \subset X \times X
\end{align*}
and let $\Delta$ be the diagonal of $X \times X$.  The algebraic intersection number $I(\Delta,graph(f))$ is an invariant of the homotopy class of $f$, called the {\bf (global) Lefschetz number} of $f$ and it is denoted $L(f)$.  As in \cite{Bott-Tu}, this can be alternatively described by
\begin{align}\label{Lf}
L(f)=\sum_{i\geq 0}(-1)^{i}trace(f_{*}^{(i)}),
\end{align}
where $f_{*}^{(i)}$ is the matrix induced by $f$ acting on $H_{i}(X)=H_{i}(X;\mathbb{R})$.  The Euler characteristic is the self-intersection number of the diagonal $\Delta$ in $X\times X$
\begin{align*}
\chi(X)=I(\Delta,\Delta)=L(id).
\end{align*}

As seen in \cite{GP}, if $f$ has isolated fixed points, we can compute the {\bf local Lefschetz number} of $f$ at a fixed point $x$ in local coordinates as
\begin{align*}
L_{x}(f)=\deg \left( z \mapsto \frac{f(z)-z}{|f(z)-z|} \right),
\end{align*}
where $z$ is on a boundary of a small disk centered at $x$ which contains no other fixed points.  Moreover we can compute the Lefschetz number by summing the local Lefschetz numbers of fixed points
\begin{align*}
L(f)=\sum_{f(x)=x}L_{x}(f).
\end{align*}

This description of $L_{x}(f)$ is given for smooth $f$ in \cite{GP}, but it is equally valid for continuous $f$ since such a map is approximated by smooth maps.  We will be computing the Lefschetz number of a homeomorphism $f:S_{g,n}\rightarrow S_{g,n}$, \textit{ignoring the marked points}.

\begin{proposition}\label{neglef}
If a homeomorphism $f:S_{g,n}\rightarrow S_{g,n}$ is homotopic (not necessarily fixing the marked points) to the identity or a multitwist, then $L(f)=\chi(S_{g,0})=2-2g$.
\end{proposition}

A {\it multitwist} is a composition of powers of Dehn twists on pairwise disjoint simple essential closed curves.

\begin{proof}
If $f$ is homotopic to the identity, the homotopy invariance of the Lefschetz number tells us $L(f)=L(id)=I(\Delta,\Delta)$ which is $\chi(S_{g,0})$.

Suppose $f$ is homotopic to a multitwist.  We will use (\ref{Lf}) to compute $L(f)$.  Note that $H_{i}(S_{g,0})$ is $0$ for $i\geq3$, $H_{0}(S_{g,0})\cong H_{2}(S_{g,0})\cong \mathbb{R}$ and $f_{*}^{(i)}$ is the identity when $i=0$ or $2$, so this implies $L(f)=2-trace(f_{*}^{(1)})$.  

There exists a set $\{\gamma_{i}\}_{i=1}^{k}$ of disjoint simple essential closed curves with some integers $n_{i}\neq0$ such that
\begin{align*}
f\simeq T^{n_{1}}_{\gamma_{1}}\circ \cdots \circ T^{n_{k}}_{\gamma_{k}},
\end{align*}
where $T^{n_{i}}_{\gamma_{i}}$ is the $n_{i}$-th power of a Dehn twist along $\gamma_{i}$.

For any curve $\gamma$,
\begin{align*}
T^{n_{i}}_{\gamma_{i}*}([\gamma])=[\gamma]+n_{i}\langle \gamma,\gamma_{i} \rangle [\gamma_{i}],
\end{align*}
where $[\gamma]$ is the homology class of $\gamma$ and $\langle \gamma,\gamma_{i} \rangle$ is the algebraic intersection number of $[\gamma]$ and $[\gamma_{i}]$.  If any $\gamma_{i}$ is a separating curve, then $[\gamma_{i}]$ is the trivial homology class and $T^{n_{i}}_{\gamma_{i}*}$ acts trivially on $H_{1}(S_{g,0})$.  We may therefore assume that each $\gamma_{i}$ is nonseparating.  After renaming the curves, we can assume that there is a subset $\{ \gamma_{1},\gamma_{2},\cdots,\gamma_{s}\}$ such that $\widehat{\gamma}=\bigcup_{i=1}^{s}\gamma_{i}$ is nonseparating and $\widehat{\gamma}\cup \gamma_{j}$ is separating for all $j>s$.  Thus, for all $k\geq j>s$
\begin{align*}
[\gamma_{j}]=\sum_{i=1}^{s}c_{ji}[\gamma_{i}],
\end{align*}
for some constants $c_{ji}\in \mathbb{R}$.  We can extend $\{[\gamma_{i}]\}_{i=1}^{s}$ to a basis of $H_{1}(S_{g,0})$,
\begin{align*}
\{\alpha_{1},\alpha_{2},\cdots,\alpha_{g},\beta_{1},\beta_{2},\cdots,\beta_{g}\}
\end{align*}
where $[\gamma_{i}]=\alpha_{i}$ for $i\leq s \leq g$ and $\langle \alpha_{i},\beta_{j} \rangle=\delta_{ij}$, $\langle \alpha_{i},\alpha_{j} \rangle=\langle \beta_{i},\beta_{j} \rangle=0$.

First suppose $s=k$, then $\langle \alpha_{j},\gamma_{i} \rangle=\langle \alpha_{j},\alpha_{i} \rangle=0$ for all $i$ and $j$.  Therefore, for all $j$
\begin{align*}
f_{*}^{(1)}(\alpha_{j})&=\alpha_{j} \\
f_{*}^{(1)}(\beta_{j})&=\beta_{j}+\sum_{i=1}^{k}n_{i}\langle \beta_{j},\gamma_{i} \rangle [\gamma_{i}]=\beta_{j}+\sum_{i=1}^{k}n_{i}\langle \beta_{j},\alpha_{i} \rangle \alpha_{i}=\beta_{j}-n_{j}\alpha_{j}.
\end{align*}
So we have
\begin{align*}
f_{*}^{(1)}=
\left( \begin{array}{c|c}
 I_{g\times g} &*\\  \hline
 0 & I_{g\times g}\\
  \end{array} \right)
\end{align*}
and $L(f)=2-trace(f_{*}^{(1)})=2-2g$.

For $s<k$, we will have
\begin{align*}
f_{*}^{(1)}(\alpha_{j})&=\alpha_{j}+\sum_{i=1}^{k}n_{i}\langle \alpha_{j},\gamma_{i} \rangle [\gamma_{i}]\\
             &=\alpha_{j}+\sum_{i=1}^{s}n_{i}\langle \alpha_{j},\alpha_{i} \rangle \alpha_{i}+\sum_{i=s+1}^{k}n_{i}\langle \alpha_{j},\gamma_{i} \rangle [\gamma_{i}]\\
             &=\alpha_{j}+\sum_{i=s+1}^{k}n_{i}\sum_{t=1}^{s}c_{it}\langle \alpha_{j},\gamma_{t} \rangle[\gamma_{t}]\\
             &=\alpha_{j}+\sum_{i=s+1}^{k}n_{i}\sum_{t=1}^{s}c_{it}\langle \alpha_{j},\alpha_{t} \rangle\alpha_{t}\\
             &=\alpha_{j}
\end{align*}
and
\begin{align*}
f_{*}^{(1)}(\beta_{j})&=\beta_{j}+\sum_{i=1}^{k}n_{i}\langle \beta_{j},\gamma_{i} \rangle [\gamma_{i}]\\
            &=\beta_{j}+\sum_{i=1}^{s}n_{i}\langle \beta_{j},\gamma_{i} \rangle [\gamma_{i}]+\sum_{i=s+1}^{k}n_{i}\sum_{t=1}^{s}c_{it}\langle \beta_{j},\gamma_{t} \rangle[\gamma_{t}]\\
            &=\beta_{j}+\sum_{i=1}^{s}n_{i}\langle \beta_{j},\alpha_{i} \rangle \alpha_{i}+\sum_{i=s+1}^{k}n_{i}\sum_{t=1}^{s}c_{it}\langle \beta_{j},\alpha_{t} \rangle\alpha_{t}\\
            &=\begin{cases}\beta_{j}, \text{ if } j>s,\\ \beta_{j}-n_{j}\alpha_{j}-\sum_{i=s+1}^{k}n_{i}c_{ij}\alpha_{j}, \text{ if } j\leq s.\end{cases}
\end{align*}
Therefore, the diagonal of the matrix $f_{*}^{(1)}$ is still all $1$'s and $L(f)=2-trace(f_{*}^{(1)})=2-2g$.
\end{proof}

\section{Bounding the dilatation from below}\label{lowerbd}

\begin{lemma}\label{intsect}
For any pseudo-Anosov element $f \in \Mod(S_{g,n})$ equipped with a Markov partition, if $L(f)<0$, then there is a rectangle $R$ of the Markov partition, such that the interiors of $f(R)$ and $R$ intersect.
\end{lemma}
\begin{proof}
Since $f$ is a pseudo-Anosov homeomorphism, it has isolated fixed points.  Suppose $x$ is an isolated fixed point of $f$ such that one of the following happens:
\begin{enumerate}
\item $x$ is a nonsingular fixed point and the local transverse orientation of ${\cal F}^{s}$ is reversed.
\item $x$ is a singular fixed point and no separatrix of ${\cal F}^{s}$ emanating from $x$ is fixed.
\end{enumerate}

\begin{claim}
$L_{x}(f)=+1$.
\end{claim}

Let $B$ be a small disk centered at $x$ containing no other fixed point of $f$.
First we show that (in local coordinates) for every $z\in \partial B$, $f(z)-z \neq \alpha z$ for all $\alpha>0$.

It is easy to verify this in case 1 by choosing local coordinates $(\xi_{1},\xi_{2})$ around $x$ so that $f$ is given by
\begin{align*}
f(\xi_{1},\xi_{2})=(-\lambda\xi_{1},\frac{-1}{\lambda}\xi_{2}).
\end{align*}

In case 2, we choose local coordinates around $x$ such that the separatrices of ${\cal F}^{s}$ emanating from $x$ are sent to rays from $0$ through the $k$th roots of unity in $\mathbb{R}^{2}$.  This means $f$ rotates each of the sectors bounded by these rays through an angle $\frac{2\pi j}{k}$ for some $j=1,\cdots,k-1$, and so for all $z\in \partial B$ $f(z)-z \neq \alpha z$  for all $\alpha>0$.

Define a smooth map $h_{0}: \partial B \rightarrow  S^{1}$ by $h_{0}(z)= \frac{f(z)-z}{|f(z)-z|}$, so $L_{x}(f)=\deg(h_{0})$ by definition.  Let $g:\partial B \rightarrow  S^{1}$ be defined by $g(z)= \frac{z}{|z|}$ and $h_{1}:S^{1} \rightarrow S^{1}$ be defined by $h_{1}(\frac{z}{|z|})=\frac{f(z)-z}{|f(z)-z|}$, so that $h_{0}=h_{1}g$.  Then
\begin{align*}
L_{x}(f)=\deg(h_{0})=\deg(h_{1}g)=\deg(h_{1})\deg(g)=\deg(h_{1})
\end{align*}
since $\deg(g)=1$.  Note that $h_{1}$ has no fixed point since for all $z\in \partial B$, $f(z)-z \neq \alpha z$ for all $\alpha>0$.  Therefore $L_{x}(f)=\deg(h_{1})=(-1)^{(1+1)}=+1$.

The assumption of $L(f)<0$ implies that there exists a fixed point $x$ of $f$ which is in neither of the cases above.  In other words, it falls into one of the cases in Figure 1.  As seen in Figure 1, there is a rectangle $R$ of the Markov partition such that the interiors of $f(R)$ and $R$ intersect.
\begin{figure}[htbp]
\begin{center}
\psfrag{.}{.}
\psfrag{x}{$x$}
\includegraphics[width=1\textwidth]{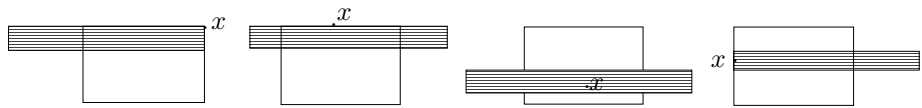}
\caption{The intersection of $f(R)$ and $R$.  $R$ is the underlying rectangle and $f(R)$ is the shaded rectangle.}
\end{center}
\end{figure}
\end{proof}

Let $\Gamma_{S}(3)\vartriangleleft \Mod(S)$ denote the kernel of the action on $H_{1}(S; \mathbb{Z}/3\mathbb{Z})$, where $S=S_{g,0}$.  In \cite{Iv2}, it is shown that $\Gamma_{S}(3)$ consists of pure mapping classes.  Setting $\Theta(g)=[\Mod(S) : \Gamma_{S}(3)]$, we conclude the following.
\begin{lemma}\label{theta}
Let $f\in \Mod(S_{g,n})$ be a pseudo-Anosov element and $\widehat{f}\in \Mod(S_{g,o})$ be the induced mapping class obtained by forgetting marked points.  There exists a constant $1\leq \alpha \leq \Theta(g)$ such that $\widehat{f}^{\alpha}$ satisfies exactly one of the following:
\begin{enumerate}
\item $\widehat{f}^{\alpha}$ restricts to a pseudo-Anosov map on a connected subsurface.
\item $\widehat{f}^{\alpha}=$Id.
\item $\widehat{f}^{\alpha}$ is a multitwist map.
\end{enumerate}
\end{lemma}

\begin{remark}
For the first two cases of Lemma \ref{theta}, one can find $\alpha$ bounded by a linear function of $g$, but in case 3, $\alpha$ may be exponential in $g$.
\end{remark}

\begin{theorem}\label{mainthm}
For $g\geq 2$, given any pseudo-Anosov $f \in \Mod(S_{g,n})$, let $1\leq \alpha\leq \Theta(g)$ be as in Lemma \ref{theta}.  Then
\begin{align*}
\log \lambda(f) \geq \min \left\{ \frac{\log2}{\alpha(12g-12)} \text{ , } \frac{\log(18g+6n-18)}{2\alpha(18g+6n-18)} \right\}.
\end{align*}

\end{theorem}

\begin{proof}
We will deal with case 1 of Lemma \ref{theta} first.

If $\widehat{f}^{\alpha}$ restricts to a pseudo-Anosov homeomorphism on a connected subsurface $\sum_{g_{0},n_{0}}$ of $S_{g,0}$ of genus $g_{0}$ with $n_{0}$ boundary components (we have $2g_{0}+n_{0} \leq 2g$), then Penner's lower bound tells us
\begin{align*}
\lambda(\widehat{f}^{\alpha})\geq \frac{\log2}{12g_{0}-12+4n_{0}} \geq \frac{\log2}{12g-12}.
\end{align*}
Hence $\log \lambda(f)\geq \log \lambda(\widehat{f})> \frac{\log2}{\alpha(12g-12)}$.

If $\widehat{f}^{\alpha}$ is homotopic to the identity or a multitwist map, from Proposition \ref{neglef}, we have $L(f^{\alpha})=L(\widehat{f}^{\alpha})=\chi(S_{g,0})=2-2g<0$.  Theorem \ref{markov} tells us that for any pseudo-Anosov $f$ there is a Markov partition with $k$ rectangles, where $k\leq-9 \chi (S)-3n$.  Recall that the transition matrix $M$ obtained from the rectangles is a $k \times k$ Perron-Frobenius matrix and the Perron-Frobenius eigenvalue $\mu(M)$ equals $\lambda(f)$.

By Lemma \ref{intsect}, there is a rectangle $R$ such that the interiors of $f^{\alpha}(R)$ and $R$ intersect. This implies that there is a nonzero entry on the diagonal of $M^{\alpha}$.  Applying Proposition \ref{lbd}, we obtain that $\mu((M^{\alpha})^{2k})=\mu(M^{2k\alpha})$ is at least $k$, so we have $(\lambda(f))^{2k\alpha}=\lambda(f^{2k\alpha})=\mu(M^{2k\alpha}) \geq k$.

One can easily check $\frac{\log x}{x}$ is monotone decreasing for $x\geq 3$.  Since $3\leq k \leq 18g+6n-18$, we have
\begin{align*}
\log \lambda(f) \geq \frac{\log k}{2\alpha k} \geq \frac{\log(18g+6n-18)}{2\alpha(18g+6n-18)}.
\end{align*}
\end{proof}

\begin{remark}
Penner's proof in \cite{Pe} does not use Lefschetz numbers which we used to conclude that $\mu(M^{2k\alpha})$ is at least $k$, so we obtain a sharper lower bound for $n\gg g$.
\end{remark}

\section{An example which provides an upper bound}\label{upperbd}

\subsection{For the case genus $g=2$}

In this section, we will construct a pseudo-Anosov $f \in \Mod(S_{2,n})$ for all $n\geq 31$ then we compute its dilatation which gives us an upper bound for $l_{2,n}$.

Let $S_{0,m+2}$ be a genus $0$ surface with $m+2$ marked points (i.e. a marked sphere), and we recall an example of pseudo-Anosov $\phi \in \Mod(S_{0,m+2})$ in \cite{HK}.  We view $S_{0,m+2}$ as a sphere with $s+1$ marked points $X$ circling an unmarked point $x$ and $t+1$ marked points $Y$ circling an unmarked point $y$, and a single extra marked point $z$.  We can also draw this as a ``turnover'', as in Figure \ref{turnover}.  Note that $|X \cap Y|=1$, $|X|=s+1$, $|Y|=t+1$ and $m=s+t$.
\begin{figure}[htbp]
\begin{center}
\psfrag{y}{$y$}
\psfrag{x}{$x$}
\psfrag{z}{$z$}
\psfrag{=}{$\cong$}
\includegraphics[width=0.9\textwidth]{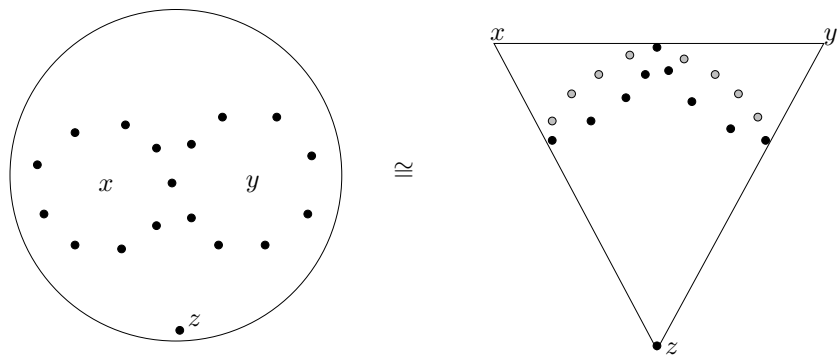}
\caption{Two way of viewing a marked sphere.  Black dots are marked points and the shaded dots on the right are marked points at the back.}\label{turnover}
\end{center}
\end{figure}

We define homeomorphisms $\alpha_{s}$, $\beta_{t}:S_{0,m+2}\rightarrow S_{0,m+2}$ such that $\alpha_{s}$ rotates the marked points of $X$ counterclockwise around $x$ and $\beta_{t}$ rotates the marked points of $Y$ clockwise around $y$; see Figure \ref{phi}.  Define $\phi_{s,t} := \beta_{t}\alpha_{s}$. 
\begin{figure}[htbp]
\begin{center}
\psfrag{y}{$y$}
\psfrag{x}{$x$}
\psfrag{z}{$z$}
\psfrag{f1}{$\alpha_{s}$}
\psfrag{f2}{$\beta_{t}$}
\includegraphics[width=0.9\textwidth]{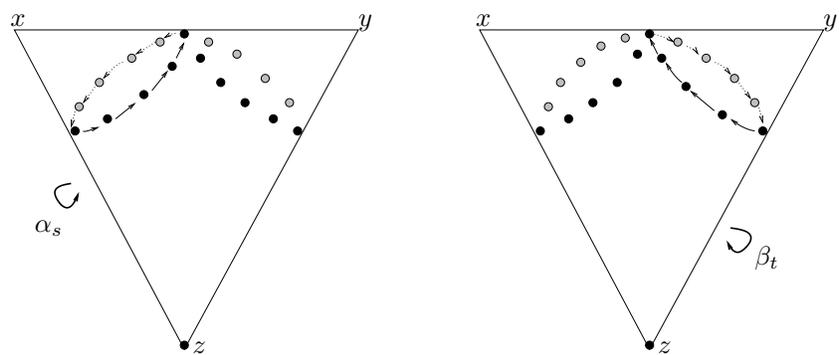}
\caption{Homeomorphisms $\alpha_{s}$ and $\beta_{t}$.}\label{phi}
\end{center}
\end{figure}
In \cite{HK}, it is shown that $\phi_{s,t}$ is pseudo-Anosov by checking it satisfies the criterion of \cite{BH}.  We also note that from this one can check that $x$, $y$ and $z$ are fixed points of a pseudo-Anosov representative of $\phi_{s,t}$. Moreover, for $s$, $t \geq 1$ the dilatation of $\phi_{s,t}$ equals the largest root of the polynomial
\begin{align*}
T_{s,t}(x)&=x^{t+1}(x^{s}(x-1)-2)+x^{s+1}(x^{-s}(x^{-1}-1)-2) \\
           &=(x-1)x^{(s+t+1)}-2(x^{s+1}+x^{t+1})-(x-1).
\end{align*}
The dilatation is minimized when $s=\lfloor \frac{m}{2}\rfloor$ and $t= \lceil \frac{m}{2}\rceil$.  Let us define $\phi:=\phi_{\lfloor \frac{m}{2}\rfloor \text{,} \lceil \frac{m}{2}\rceil}$ and its dilatation is the largest root of the polynomial 
\begin{align*}
T_{m}(x)&:=T_{\lfloor \frac{m}{2}\rfloor \text{,} \lceil\frac{m}{2}\rceil}(x)\\
&=(x-1)x^{(m+1)}-2\left( x^{\lfloor \frac{m}{2}\rfloor+1}+x^{\lceil\frac{m}{2}\rceil+1}\right)-(x-1).
\end{align*}

\begin{proposition}\label{lroot}
If $m\geq 5$, then the largest real root of $T_{m}(x)$ is bounded above by $m^{\frac{3}{m}}$.
\end{proposition}
\begin{proof}
For all $m$, we have $T_{m}(1)=-4$.  It is sufficient to show that for all $x\geq m^{\frac{3}{m}}$, we have $T_{m}(x)>0$.  Dividing the inequality by $x^{(m+1)}$, it is equivalent to show
\begin{align*}
(x-1)+x^{-(m+1)}>2\left( x^{\lfloor \frac{m}{2}\rfloor-m}+x^{\lceil\frac{m}{2}\rceil-m}\right)+x^{-m}.
\end{align*}
For $m\geq5$, one can verify the following inequalities hold for all $x\geq m^{\frac{3}{m}}$
\begin{enumerate}
\item $x-1 >\frac{3 \log m}{m}\geq \frac{9}{2m}$,
\item $x^{\lfloor \frac{m}{2}\rfloor-m}\leq x^{\lceil\frac{m}{2}\rceil-m}\leq \frac{1}{m}$,
\item $x^{-m}\leq \frac{1}{25m}$.
\end{enumerate}
Therefore,
\begin{align*}
(x-1)+x^{-(m+1)}&>x-1>\frac{9}{2m}>\frac{101}{25 m}=2(\frac{1}{m}+\frac{1}{m})+\frac{1}{25m}\\
&\geq 2\left( x^{\lfloor \frac{m}{2}\rfloor-m}+x^{\lceil\frac{m}{2}\rceil-m}\right)+x^{-m}.
\end{align*}
\end{proof}

\begin{remark}
Proposition \ref{lroot} fails if we try to replace the bound with $c^{\frac{1}{m}}$ where $c$ is any constant.
\end{remark}
\begin{remark}
$\phi_{s,t}$ is not the example which gives the best upper bound on $l_{0,m}$ in \cite{HK}.  In fact they show $\log \lambda(\phi_{s,t})$ is strictly greater than $l_{0,m}$ for $m\geq 8$, for all $s$ and $t$.
\end{remark}

Next, we take a cyclic branched cover $S_{2,n}$ of $S_{0,m+2}$ with branched points $x$, $y$, $z$ where $n=5(m+1)+1$.  See Figure \ref{cover}.  Define $\widetilde{X}=\{ \text{marked points around } \widetilde{x} \}$ and $\widetilde{Y}=\{ \text{marked points around } \widetilde{y} \}$, so we have $|\widetilde{X} \cap \widetilde{Y}|=5$, $|\widetilde{X}|=5(s+1)$ and $|\widetilde{Y}|=5(t+1)$.
\begin{figure}[htbp]
\begin{center}
\psfrag{y}{$\widetilde{y}$}
\psfrag{x}{$\widetilde{x}$}
\psfrag{z}{$\widetilde{z}$}
\psfrag{y1}{$y$}
\psfrag{x1}{$x$}
\psfrag{z1}{$z$}
\psfrag{pi}{$\pi$}
\includegraphics[width=0.9\textwidth]{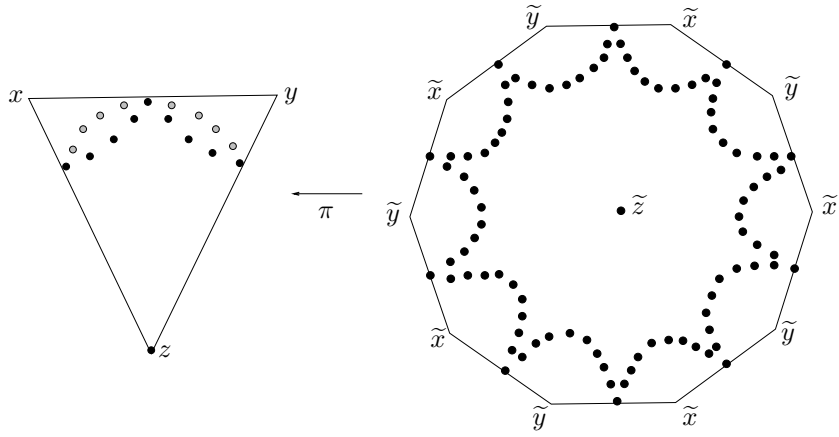}
\caption{$\pi$ is the covering map.  To form $S_{2,n}$ from the decagon, identify the opposite sides.  Then $\pi$ is the quotient by the group generated by rotation of an angle $2\pi/5$.}\label{cover}
\end{center}
\end{figure}

We lift $\alpha_{s}$, $\beta_{t}$ to $S_{2,n}$ and call them $\widetilde{\alpha_{s}}$, $\widetilde{\beta_{t}}$, so that $\widetilde{\alpha_{s}}$ rotates the marked points of $\widetilde{X}$ counterclockwise around $\widetilde{x}$ and $\widetilde{\beta_{t}}$ rotates the marked points of $\widetilde{Y}$ clockwise around $\widetilde{y}$; see Figure \ref{psi}.  We define $\psi_{s,t}:=\widetilde{\beta_{t}}\widetilde{\alpha_{s}}$.  It follows that $\psi_{s,t}$ is a lift of $\phi_{s,t}$, and so is pseudo-Anosov with $\lambda(\psi_{s,t})=\lambda(\phi_{s,t})$.  An invariant train track for $\psi_{s,t}$ is obtained by lifting the one constructed in \cite{HK}, and is shown in Figure \ref{oldtt} for $s=t=3$.
\begin{figure}[htbp]
\begin{center}
\psfrag{y}{$\widetilde{y}$}
\psfrag{x}{$\widetilde{x}$}
\psfrag{z}{$\widetilde{z}$}
\psfrag{f1}{$\widetilde{\alpha_{s}}$}
\psfrag{f2}{$\widetilde{\beta_{t}}$}
\includegraphics[width=1\textwidth]{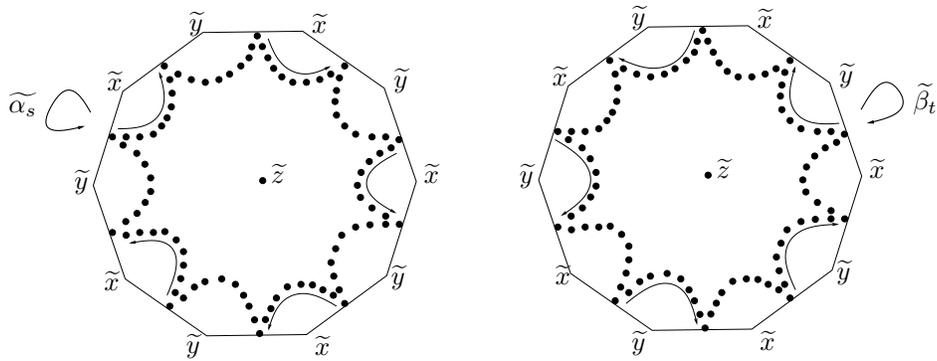}
\caption{Homeomorphisms $\widetilde{\alpha_{s}}$ and $\widetilde{\beta_{t}}$.}\label{psi}
\end{center}
\end{figure}
\begin{figure}[htbp]
\begin{center}
\psfrag{y}{$\widetilde{y}$}
\psfrag{x}{$\widetilde{x}$}
\psfrag{z}{$\widetilde{z}$}
\includegraphics[width=0.7\textwidth]{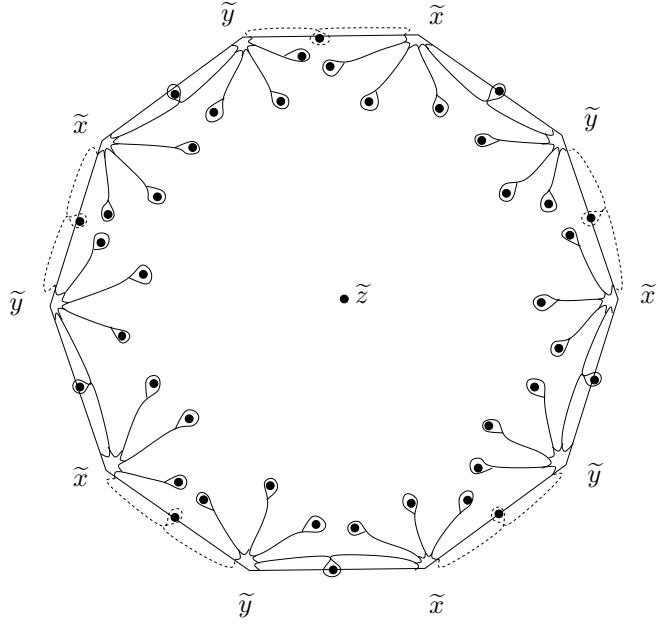}
\caption{A train track for $\psi_{3,3}$}\label{oldtt}
\end{center}
\end{figure}

Hence for $n=5(m+1)+1\geq 31$, we have constructed a pseudo-Anosov $\psi=\psi_{\lfloor \frac{m}{2}\rfloor \text{,} \lceil \frac{m}{2}\rceil} \in \Mod(S_{2,n})$ with $\lambda(\psi) = \lambda(\phi) \leq m^{\frac{3}{m}}$ which implies
\begin{align*}
\log\lambda(\psi) \leq \frac{3\log m}{m}=\frac{15\log(n-6)-15\log 5}{n-6}.
\end{align*}

We will now extend $\psi$ so that $n$ can be an arbitrary number $\geq 31$.  We add an extra marked point $p_{1}$ on $S_{2,n}$ between points in $\widetilde{X}$ or $\widetilde{Y}$ {\it except the places shown in Figure \ref{except}}.
\begin{figure}
\begin{center}
\psfrag{y}{$\widetilde{y}$}
\psfrag{x}{$\widetilde{x}$}
\psfrag{z}{$\widetilde{z}$}
\includegraphics[width=0.9\textwidth]{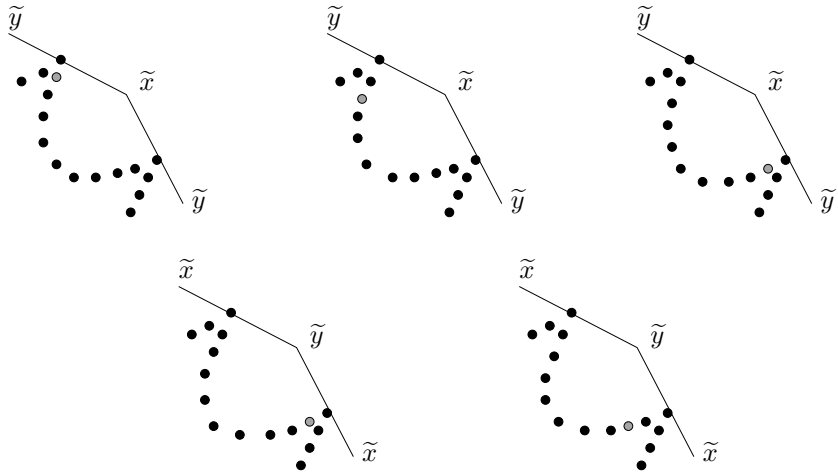}
\caption{We are \textit{not} allowed to add $p_{1}$ in the places indicated by a shaded point.}\label{except}
\end{center}
\end{figure}

Without loss of generality we assume $p_{1}$ is added in $\widetilde{X}$ to obtain $S_{2,n+1}$ and we define $\psi_{1}:=\widetilde{\beta_{t}}\widetilde{\alpha_{s}}'\in\Mod(S_{2,n+1})$ where $\widetilde{\alpha_{s}}'$ is extended from $\widetilde{\alpha_{s}}$ in the obvious way; see Figure \ref{newpsi}.
\begin{figure}[htbp]
\begin{center}
\psfrag{y}{$\widetilde{y}$}
\psfrag{x}{$\widetilde{x}$}
\psfrag{z}{$\widetilde{z}$}
\psfrag{p}{$p_{1}$}
\psfrag{f1}{$\widetilde{\alpha_{s}}'$}
\includegraphics[width=0.9\textwidth]{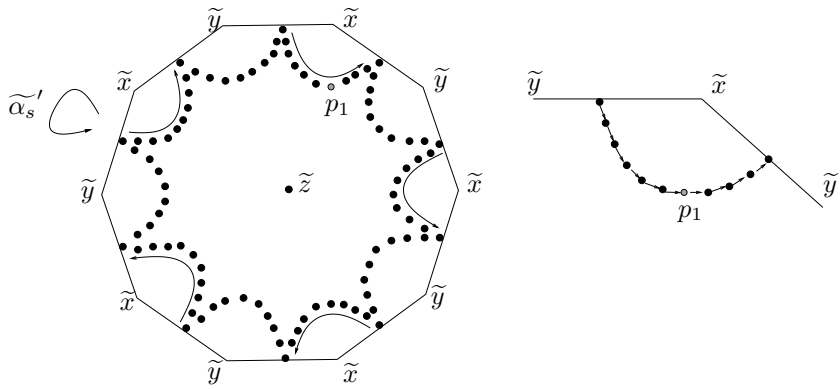}
\caption{The homeomorphism $\widetilde{\alpha_{s}}'$.  The figure on the right is a local picture near the added point $p_{1}$.}\label{newpsi}
\end{center}
\end{figure}
One can check that $\psi_{1}$ is pseudo-Anosov via the techniques of \cite{BH}.  An invariant train track for $\psi_{1}$ is shown in Figure \ref{newtt} and is obtained by modifying the invariant train track for $\psi$ shown in Figure \ref{oldtt}.

\begin{figure}[htbp]
\begin{center}
\psfrag{y}{$\widetilde{y}$}
\psfrag{x}{$\widetilde{x}$}
\psfrag{z}{$\widetilde{z}$}
\includegraphics[width=0.7\textwidth]{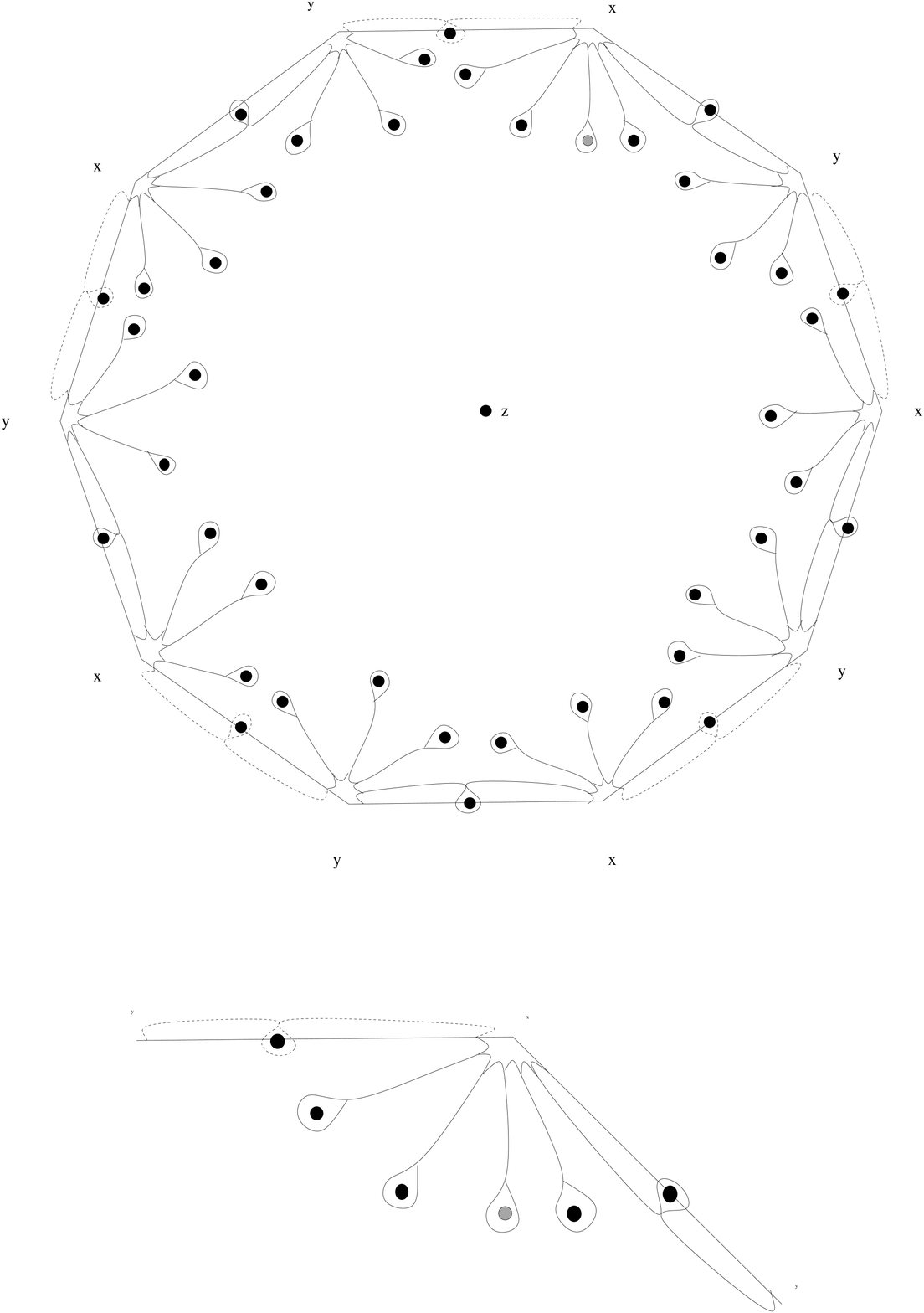}
\caption{A train track for $\psi_{1}$.  The figure on the bottom is a local picture.}\label{newtt}
\end{center}
\end{figure}
Next, we will show $\lambda(\psi_{1})\leq \lambda(\psi)$.  Let $H$ (respectively $H_{1}$) be the associated transition matrix of the train track map for $\psi$ (respectively $\psi_{1}$), and let $\Gamma$ (respectively $\Gamma_{1}$) be the induced directed graph as constructed in Section \ref{secMarkov}.

From the construction above (i.e. adding $p_{1}$), the directed graph $\Gamma_{1}$ is obtained by adding a vertex on the edge going out from some vertex $i$ in $\Gamma$ (that is, subdividing the edge going out from $i$) where $i$ has exactly one edge coming in and exactly one edge going out.  This implies $P_{\Gamma_{1}}(i,k+1)=P_{\Gamma}(i,k)$ and
\begin{align*}
\sqrt[k+1]{P_{\Gamma_{1}}(i,k+1)}\leq \sqrt[k]{P_{\Gamma_{1}}(i,k+1)}=\sqrt[k]{P_{\Gamma}(i,k)}
\end{align*}
for all $k$.  Since $H$ and $H_{1}$ are Perron-Frobenius matrices with Perron-Frobenius eigenvalues corresponding to the dilatations of $\psi$ and $\psi_{1}$, and Proposition \ref{dgraph} tells us $\mu(H_{1}) \leq \mu(H)$, we have $\lambda(\psi_{1})=\mu(H_{1})$ is no greater than $\lambda(\psi)=\mu(H)$.  

We can obtain $\psi_{2}$, $\psi_{3}$ and $\psi_{4}$ by repeating the construction above of adding more marked points without increasing dilatations (i.e. $\lambda(\psi_{c}) \leq \lambda(\psi)$ for $c=1,2,3,4$).  Since $\frac{\log m}{m}\geq \frac{\log (m+1)}{m+1}$, we need not consider the cases with $c\geq5$.  Therefore, set $f:S_{2,n}\rightarrow S_{2,n}$ to be $\psi_{c}$, where $n=5(m+1)+1+c$ with $c<5$, and where $\psi_{0}=\psi$.  For $n\geq 31$, we have
\begin{align*}
\log \lambda(f) &\leq \log \lambda(\psi) < \frac{3\log m}{m} < \frac{3\log\left( \frac{n-11}{5}\right)}{\left( \frac{n-11}{5}\right)},
\end{align*}
where $m=\left\lfloor \frac{n-6}{5} \right\rfloor$.
\begin{theorem}\label{genus2}
There exists $\kappa_{2}>0$ such that $l_{2,n}<\frac{\kappa_{2}\log n}{n}$, for all $n\geq3$.
\end{theorem}

\begin{proof}
From the discussion above, for $n\geq31$,
\begin{align*} 
l_{2,n}<\frac{3\log\left( \frac{n-11}{5}\right)}{\left( \frac{n-11}{5}\right)}
       <\frac{\kappa_{2}'\log n}{n},
\end{align*}
for some $\kappa_{2}'$.  For $3\leq n \leq 30$, let $\kappa_{2}''= \max\{l_{2,3}, l_{2,4}, \cdots, l_{2,30}\}$ then
\begin{align*}
l_{2,n} \leq \kappa_{2}'' = \left(\kappa_{2}''\frac{31}{\log31}\right) \frac{\log31}{31} < \left( \kappa_{2}''\frac{31}{\log31}\right)\frac{\log n}{n}.
\end{align*}
Let $\kappa_{2}:=\max\{\kappa_{2}',\kappa_{2}''\frac{31}{\log31}\}$.
\end{proof}

\subsection{Higher genus cases}

We can generalize our construction and extend to any genus $g>2$.  For any fixed $g>2$, we define $\psi$ to be a homeomorphism of $S_{g,n}$ in the same fashion with $n=(2g+1)(m+1)+1$ by taking an appropriate branched cover over $S_{0,m+2}$, and we can again extend to arbitrary $n$ by adding $c$ extra marked points and constructing $\psi_{c}$.  Define $f:S_{g,n}\rightarrow S_{g,n}$ to be $\psi_{c}$ where $n=(2g+1)(m+1)+1+c$.  If $n\geq 6(2g+1)+1$, then
\begin{align*}
\log \lambda(f) &< \frac{3\log m}{m}\text{, where } m=\left\lfloor \frac{n-1}{2g+1}\right\rfloor-1\\
                &<\frac{3\log\left( \frac{n-4g-3}{2g+1}\right)}{\left( \frac{n-4g-3}{2g+1}\right)}.
\end{align*}

\begin{theorem}\label{upbd}
For any fixed $g\geq 2$, there exists $\kappa_{g}>0$ such that $l_{g,n} < \frac{\kappa_{g}\log n}{n}$, for all $n\geq3$.
\end{theorem}
\begin{proof}
This is similar to the proof of Theorem \ref{genus2}, where $\kappa_{g}:=\max\{\kappa_{g}',\kappa_{g}''\frac{12g+7}{\log(12g+7)}\}$.
\end{proof}

\begin{proof}[Proof of Theorem \ref{main}]
We only need to prove that the lower bounds on $\log \lambda(f)$ of Theorem \ref{mainthm} are bounded below by $\frac{\log n}{\omega_{g}n}$ for some $\omega_{g}$ depending only on $g$, then let $c_{g}=\max\{\kappa_{g},\omega_{g}\}$.  We use the monotone decreasing property of $\frac{\log n}{n}$ for $n\geq3$.  Let
\begin{align*}
\omega'_{g}(\alpha) := \frac{\alpha(12g-12)}{\log2}\frac{\log 3}{3} \geq \frac{\alpha(12g-12)}{\log2}\frac{\log n}{n}
\end{align*}
and so
\begin{align*}
\frac{\log2}{\alpha(12g-12)} \geq \frac{\log n}{\omega'_{g}(\alpha)n}.
\end{align*}
For $n\geq g-1$,
\begin{align*}
\frac{\log(18g+6n-18)}{2\alpha(18g+6n-18)} \geq\frac{\log24n}{2\alpha24n}>\frac{1}{48\alpha}\frac{\log n}{n}.
\end{align*}
For $3\leq n< g-1$,
\begin{align*}
\frac{\log(18g+6n-18)}{2\alpha(18g+6n-18)} > \frac{\log(24(g-1))}{2\alpha 24(g-1)}\geq\frac{\log(24(g-1))}{2\alpha 24(g-1)}\frac{3}{\log 3}\frac{\log n}{n}.
\end{align*}
Let $\omega_{g}:=\max \{ \omega'_{g}(\alpha),48\alpha,\frac{48\alpha(g-1)\log 3}{3\log(24(g-1))} \}$, where $0\leq \alpha \leq \Theta(g)$.
\end{proof}

\section{Appendix}
We will construct a example to prove that $l_{1,2n}$ has an upper bound of the same order as Penner's lower bound in \cite{Pe}, i.e. $l_{1,2n}= O(\frac{1}{n})$.  The construction is analogous to the one given by Penner for $S_{g,0}$ in \cite{Pe}.

Let $S_{1,2n}$ be a marked torus of $2n$ marked points.  Let $a$ and $b$ be essential simple closed curves as in Figure \ref{scc}.  
\begin{figure}[htbp]
\begin{center}
\psfrag{a}{$a$}
\psfrag{b}{$b$}
\includegraphics[width=0.3\textwidth]{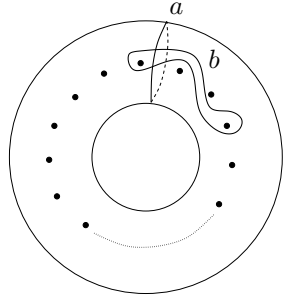}
\caption{Essential simple closed curves $a$ and $b$ on a marked torus.}\label{scc}
\end{center}
\end{figure}
Let $T_{a}^{-1}$ be the left Dehn twist along $a$ and $T_{b}$ be the right Dehn twist along $b$, then we define 
\begin{align*}
f:=\rho\circ T_{b}\circ T_{a}^{-1}\in \Mod(S_{1,2n})
\end{align*}
where $\rho$ rotates the torus clockwise by an angle of $2\pi/n$, so it sends each marked point to the one which is two to the right.  As in \cite{Pe-const}, $f^{n}$ is shown to be pseudo-Anosov, and thus so is $f$.  Figure \ref{fton} shows a bigon track for $f^{n}$.
\begin{figure}[htbp]
\begin{center}
\psfrag{a}{$a$}
\psfrag{b}{$b$}
\includegraphics[width=0.4\textwidth]{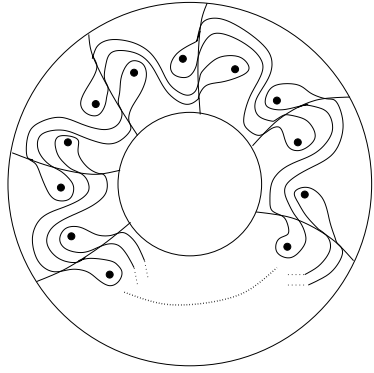}
\caption{A bigon track for $f^{n}$.}\label{fton}
\end{center}
\end{figure}

We obtain the transition matrix $M^{n}$ associated to the train track map of $f^{n}$ where $M^{n}$ is an integral Perron-Frobenius matrix and the Perron-Frobenius eigenvalues $\mu(M^{n})$ is the dilatation $\lambda(f^{n})$ of $f^{n}$.  For $n\geq 5$,
\begin{align*}
M^{n}=
\left( \begin{array}{cccccccccccccc}
 1&      1&       0&      1& 0& 0& 0& 0& \cdots& \cdots&     0& 0& 0& 0\\
 1&      2&       0&      1& 0& 0& 0& 0& \cdots&  \cdots&    0& 0& 1& 0\\
 0&      0&       1&      1& 0& 1& 0& 0& \cdots&  \cdots&    0& 0& 0& 0\\
 1&      1&       1&      3& 0& 1& 0& 0& \cdots& \cdots&     0& 0& 0& 0\\
 0&      0&       0&      0& 1& 1& 0& 1& \cdots& \cdots&     0& 0& 0& 0\\
 0&      0&  1&      1& 1& 3& 0& 1& \cdots&  \cdots&    0& 0& 0& 0\\
  0&      0&  0&      0& 0& 0& 1& 1& \cdots&  \cdots&    0& 0& 0& 0\\
  0&      0&  0& 0& 1& 1& 1& 3& \cdots&  \cdots&    0& 0& 0& 0\\
  0&      0&  0& 0& 0& 0& 0& 0& \cdots&  \cdots&    0& 0& 0& 0\\
 0&      0&  0& 0& 0& 0& 1& 1& \cdots&  \cdots&    0& 0& 0& 0\\
 0&      0&  0& 0& 0& 0& 0& 0& \cdots&  \cdots&    0& 0& 0& 0\\
 \vdots&  \vdots& \vdots& \vdots& \vdots& \vdots& \vdots& \vdots& \cdots&  \cdots&    \vdots& \vdots& \vdots& \vdots\\
 0&      0&  0& 0& 0& 0& 0& 0& \cdots&  \cdots&    0& 0& 0& 0\\
 0&      0&  0& 0& 0& 0& 0& 0& \cdots& \cdots&     0& 1& 0& 0\\
 0&      0&  0& 0& 0& 0& 0& 0& \cdots& \cdots&     0& 1& 0& 0\\
0&      0&  0& 0& 0& 0& 0& 0& \cdots&  \cdots&    1& 1& 0& 1\\
 0&      0&  0& 0& 0& 0& 0& 0& \cdots&  \cdots&   1& 3& 0& 1\\
 1&      2&       0&      1& 0& 0& 0& 0& \cdots& \cdots&  0& 0& 2& 1\\
 1&      2&       0&      1& 0& 0& 0& 0& \cdots& \cdots&  1& 1& 2& 3\\
  \end{array} \right)_{2n\times 2n}.
\end{align*}

Note that pairs of columns in the middle of the matrix shift down by $2$ in succession.  For $n\geq 5$, the greatest column sum of $M^{n}$ is $9$ and the greatest row sum of $M^{n}$ is $11$.  One can verify that both the greatest column sum and the greatest row sum are $\leq11$ for $0<n\leq 4$.  Therefore, for $n\geq1$
\begin{align*}
11\geq \mu(M^{n}) = \lambda(f^{n}) =(\lambda(f))^{n}\\
\Rightarrow  l_{1,2n}\leq \log \lambda(f) \leq \frac{\log 11}{n}.
\end{align*}

\bibliographystyle{amsalpha}

\bibliography{lpa}


\bigskip

\noindent Chia-yen Tsai: \newline
Department of Mathematics, University of Illinois, Urbana-Champaign, IL 61801 \newline \noindent
\texttt{ctsai6@math.uiuc.edu}

\end{document}